\documentclass{amsart}
\usepackage{amssymb,latexsym,amsfonts,amsthm,amsmath}
\usepackage{enumerate, epsfig, subfigure, wasysym}

\numberwithin{equation}{section}

\theoremstyle{plain}
\newtheorem{theorem}{Theorem}
\newtheorem{corollary}[theorem]{Corollary}
\newtheorem{lemma}[theorem]{Lemma}

\theoremstyle{definition}

\begin{document}
\title[Duality properties of strong isoperimetric inequalities]{Duality properties of strong isoperimetric inequalities on a planar graph and combinatorial curvatures}
\author{Byung-Geun Oh}
\address{Department of Mathematics Education, Hanyang University, 17 Haengdang-dong, Seongdong-gu, Seoul 133-791, Korea}
\email{bgoh@hanyang.ac.kr}
\date{\today}

\begin{abstract}
This paper is about hyperbolic properties on planar graphs. First,
we study the relations among various kinds of strong isoperimetric inequalities on planar graphs and their duals.
In particular, we show that a planar graph satisfies a strong isoperimetric inequality if and only if
its dual has the same property, if the graph satisfies some minor regularity conditions and
we choose an appropriate notion of strong isoperimetric inequalities. Second,
we consider planar graphs where negative combinatorial curvatures
dominate, and use the outcomes of the first part
to strengthen the results of Higuchi,  \.{Z}uk, and, especially, Woess. Finally, we study the relations between
Gromov hyperbolicity and strong isoperimetric inequalities on planar graphs, and give a proof that
a planar graph satisfying a proper kind of a strong isoperimetric inequality
must be Gromov hyperbolic if face degrees of the graph are bounded.
We also provide some examples to support our results.
\end{abstract}

\maketitle
\section{Introduction}
The main topic of this paper is strong isoperimetric inequalities on planar graphs, as one can guess from
the title. In fact, we have studied the relations of three kinds of strong isoperimetric inequalities on
planar graphs and their dual graphs, and as an application we have strengthened the results of \cite{Hig, Woe, Zuk}.
We believe that some of our works can be considered a \emph{`similar effort'} for showing that
a planar graph satisfies a strong isoperimetric inequality if and only of its dual has the same property,
as suggested in \cite{Woe}.

To describe our results precisely, let $G$ be a connected simple planar graph embedded into $\mathbb{R}^2$ locally finitely
such that its dual graph $G^*$ is also simple. See Section~\ref{S2} for details of the terminologies.
We denote by $V(G), E(G)$, and $F(G)$ the vertex set,
the edge set, and the face set, respectively, of $G$. For each $v \in V(G)$, $\deg (v)$ is the number of edges
with one end at $v$. Similarly for each $f \in F(G)$, $\deg(f)$ is the number of edges surrounding $f$.
In this paper we assume that $3 \leq \deg(v), \deg(f) < \infty$
for every $v \in V(G)$ and $f \in F(G)$, unless otherwise stated.

Next suppose $S$ is a finite subgraph of $G$,
and we consider three types of boundaries of $S$. The first one is $\partial S$,
the set of edges in $E(G)$ such that each element of $\partial S$ has
one end on $V(S)$ and the other end on $V(G) \setminus V(S)$.
The second one is $\partial_v S \subset V(S)$, each of whose element is an end vertex of some edges in $\partial S$.
The last boundary $\partial_e S$ is the set of edges surrounding $S$; i.e., $e \in \partial_e S$
if and only if $e \in E(S)$ and $e$ belongs to $E(f)$ for some $f \in F(G) \setminus F(S)$.
Now we define three different strong isoperimetric constants by
\begin{equation}\label{iso1}
\imath (G) := \inf_S \frac{|\partial S|}{\mbox{Vol}(S)}, \qquad
\jmath (G) := \inf_S \frac{|\partial_v S|}{|V(S)|}, \qquad
\kappa (G) := \inf_S \frac{|\partial_e S|}{|F(S)|},
\end{equation}
where $|\cdot|$ denotes the cardinality,
$\mbox{Vol}(S) = \sum_{v \in V(S)} \deg(v)$, and $S$ runs over all the nonempty finite subgraphs of $G$.

The above constants $\imath(G), \jmath(G), \kappa(G)$  characterize some properties of the edge set,
the vertex set, and the face set, respectively,
of $G$, and are discrete analogues of the Cheeger's constant \cite{Che}.
The condition $\imath(G)>0$ is of particular interest in spectral theory on graphs, since this condition is
equivalent to the positivity of the smallest eigenvalue of the negative Laplacian \cite{Dod, DK}, implying the
simple random walk on $G$ is transient. For more about this subject, see for instance
\cite{BMS, Fu, Gerl, Kel1, Moh, So1, Woe2} and the references therein. The constant $\kappa(G)$ is
essentially dealt with in the
geometric(combinatorial) group theory \cite{GH, Gro}, and it
was also investigated in \cite{HS, LPZ}. The constant $\jmath(G)$ appears in the geometric group theory as well,
and early versions of spectral theory on graphs \cite{Dod, So2}. Note that $\jmath(G)$ is quantitatively equivalent to $\imath(G)$ if vertex degrees of $G$ are bounded, and
to $\kappa(G)$ if face degrees of $G$ are bounded (cf. Theorem~\ref{T} below).

We will call a simple planar graph \emph{proper} if every face of the graph is a topological closed disk.
Now we are ready to describe our main result.

\begin{theorem}\label{T}
Suppose $G$ is a proper planar graph as described above, and $G^*$ is its dual graph. Then
\begin{enumerate}[(a)]
\item $\imath(G) > 0$ if and only if $\kappa(G^*) >0$, and $\kappa(G) > 0$ if and only if $\imath(G^*) >0$;
\item $\jmath(G) >0$ if and only if $\jmath(G^*) >0$;
\item if $\jmath(G) >0$, then $\imath(G) > 0$, $\imath(G^*) >0$, $\kappa(G) > 0$, and $\kappa(G^*) >0$.
\end{enumerate}
\end{theorem}

Of course the main part of the above theorem is (b). We believe that the part (a) is well known to experts,
but we have decided to contain a proof of it for the sake of completeness. Moreover,
it is not long. For (c), one can easily deduce it from (a) and (b).
The reason why we state our results as above, instead of emphasizing (b) alone, is because
this way could help one seeing the whole picture easily.

In Theorem~\ref{T} we did not require any upper bound for vertex degrees or face degrees of $G$, which
makes the theorem useful.
All the statements in Theorem~\ref{T} become  trivial if both vertex and face degrees of $G$ are bounded, since
in this case $G$ is roughly isometric(quasi-isometric) to its dual $G^*$,
hence (b) follows from Theorem (7.34) of \cite{So2}. (For rough isometries, see Section~\ref{SG}.)

In the course of proving Theorem~\ref{T}(b) we obtained the following result as a byproduct, which might be interesting
by itself (compare it with \cite[Reduction Lemma 2]{Woe}).

\begin{theorem}\label{T2}
Suppose $G$ is a proper planar graph such that $|V(S)| \leq C|\partial_v S|$
for every polygon $S \subset G$,
where $C$ is a constant not depending on $S$. If either $G$ is normal or  face degrees of $G$ are bounded,
then $\jmath(G) > 0$.
\end{theorem}
We call a planar graph \emph{normal} if it is proper and the intersection of every two different faces
is exactly one of the following: the empty set, a vertex, or an edge. We chose the terminology `normal' since
we adopted the first two properties of \emph{normal tilings} defined in \cite{GS}. For polygons,
they are basically finite unions of faces in $F(G)$ with simply connected interiors; for the precise definition,
see Section~\ref{S2}.

\begin{figure}[tb]
\centerline{
\hfill \subfigure[the graph $\Lambda$]{\epsfig{figure=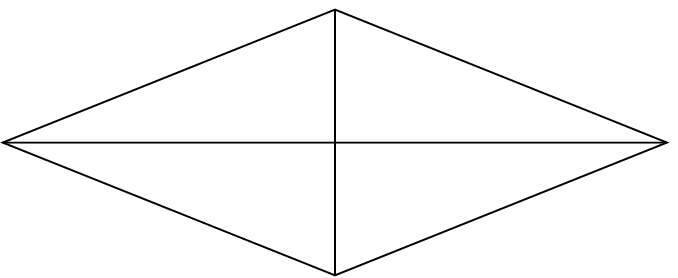, width=1.3 in, height=1.3 in}} \hfill\hfill
\subfigure[attaching eight copies of $\Lambda$ to a vertex of degree 8]{\epsfig{figure=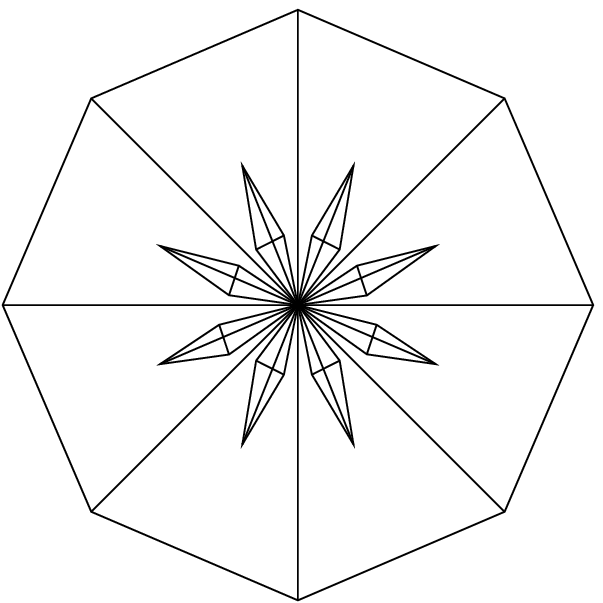, width=1.5 in, height=1.5 in}} \hfill }
\caption{constructing the graph $G$ from $\Gamma$ \label{nonproper}}
\end{figure}
One cannot omit the properness condition in Theorem~\ref{T}, because without it the statement
(b)  is no longer true.
For example, let $\Gamma$ be a triangulation of the plane such that $\deg v \geq 7$ for every $v \in V(\Gamma)$.
Then it is well known that $\jmath(\Gamma)>0$ and $\jmath(\Gamma^*)>0$. (Also see Corollary~\ref{Cor} below for this fact.)
Furthermore, let us assume that there exists a sequence of vertices $v_k \in V(\Gamma)$
such that $n_k := \deg v_k \to \infty$ as $k \to \infty$. The essential property of $\Gamma$ is
that   face degrees of $\Gamma$ are bounded since it is a triangulation of the plane, while
 vertex degrees are not bounded. Now for each $k$,
we attach $n_k$ copies of the graph $\Lambda$ in Figure~\ref{nonproper}(a)
to $v_k$ so that a degree 3 vertex of each copy is identified with $v_k$ and
each face $f \in F(\Gamma)$ with $V(f) \ni v_k$ contains exactly one copy of $\Lambda$ in it
 (Figure~\ref{nonproper}(b)). Let $G$ denote
this new graph, which is definitely not proper but satisfies all the other properties we require; that is,
$3 \leq \deg v, \deg f < \infty$ for every $v \in V(G)$ and $f \in F(G)$, and both $G$ and $G^*$ are simple.

It is not difficult to see $\jmath(G) =0$ since if we denote by $S_k$
the union of $n_k$ copies of $\Lambda$ sharing the vertex $v_k$,
then we have $\partial_v (S_k) = \{ v_k \}$ and $|V(S_k)| = 4 n_k +1$.
To see that $\jmath(G^*)>0$, first note that $G^*$ and $\Gamma^*$ are roughly isometric.
Moreover, face degrees of $G$ and $\Gamma$, or vertex degrees of $G^*$ and $\Gamma^*$, are both bounded
and we chose $\Gamma$ so that $\jmath(\Gamma^*) >0$.
Thus the inequality $\jmath(G^*) > 0$ follows from Theorem (7.34) of \cite{So2}.

To obtain an application of Theorem~\ref{T}, let us introduce so-called \emph{combinatorial curvatures}
defined on planar graphs. Suppose $G$ is a proper planar graph as before.
For each $e \in E(G)$, $v \in V(G)$, and $f \in F(G)$ we define \emph{edge curvature} $\phi$, \emph{vertex curvature}
$\psi$, and \emph{face curvature} $\chi$ by
\begin{align*}
\phi (e) & = \sum_{w \in V(e)} \frac{1}{\deg (w)} + \sum_{g: e \in E(g)} \frac{1}{\deg (g)} -1,\\
\psi (v) & = 1 - \frac{\deg(v)}{2} + \sum_{g: v \in V(g)} \frac{1}{\deg (g)}, \quad \mbox{and}\\
\chi (f) & = 1 - \frac{\deg(f)}{2} + \sum_{w \in V(f)} \frac{1}{\deg(w)}.
\end{align*}
In the above $w$ and $g$ stand for a vertex and a face, respectively.
Remark that the notations $\phi, \psi,$ and $\chi$ are those used in \cite{Woe}, but we have changed the signs.
Other than the above combinatorial curvatures for planar graphs, there is another one
 called \emph{corner curvature} \cite{Kel2, KP}.

Recently combinatorial curvatures have been extensively studied by various researchers
\cite{BP1, BP2, BMS, Cor, DM, Hig, Kel1, Kel2, KP, Sto, Woe, Zuk}, but this concept
was introduced more than seven decades ago.
In \cite[Chap. XII]{Nev}, Nevanlinna introduced a characteristic number called \emph{excess},
which is essentially equal to  the vertex curvature.
Moreover, there might be older literature in this line than \cite{Nev}, since in \cite{Nev}
Nevanlinna mentioned a work of Teichm\"{u}ller \cite{Tei} related to excess.
In fact, excess was defined for
a special type of bipartite regular planar graphs, called Speiser graphs, which capture the combinatorial properties of
meromorphic functions defined on some simply connected Riemann surfaces and ramified only over finitely
many points in the extended complex plane $\mathbb{\overline{C}}$. For more about Speiser graphs and excess, see for example
\cite{Nev, BMeS, Oh}.

For finite subsets $E \subset E(G)$,
let $\overline{\phi}(E) = (1/|E|)\sum_{e\in E} \phi(e)$ and define the \emph{upper average} of $\phi$ on $G$ by
\[
\overline{\phi}(G) = \limsup_{|E(S)| \to \infty} \overline{\phi}(E(S)),
\]
where limit superior is taken over all simply connected finite subgraphs $S$ of $G$. Similarly, for finite subsets
$V \subset V(G)$ and $F \subset F(G)$, let
$\overline{\psi}(V) = (1/|V|)\sum_{v\in V} \psi(v)$ and
$\overline{\chi}(F) = (1/|F|)\sum_{f\in F} \chi(f)$, and define the upper averages of $\psi$ and $\chi$ on $G$,
respectively, by
\[
\overline{\psi}(G) = \limsup_{|V(S)| \to \infty} \overline{\psi}(V(S)), \qquad
\overline{\chi}(G) = \limsup_{|F(S)| \to \infty} \overline{\chi}(F(S)),
\]
where the limit superiors are also taken over all simply connected finite subgraphs $S$ of $G$.
Our second main result is the following.

\begin{theorem}\label{TTT}
Suppose $G$ is a proper planar graph.
\begin{enumerate}[(a)]
\item If $\overline{\phi}(G) < 0$ or $\overline{\chi}(G) < 0$, then $\jmath(G) >0$.
\item If $\overline{\psi}(G) < 0$, and if either $G$ is normal or the vertex degrees of $G$ are bounded,
     then $\jmath(G)>0$.
\item It is possible to have $\overline{\psi}(G) < 0$ and $\jmath(G)=\kappa(G) = 0$.
\end{enumerate}
\end{theorem}

The most surprising part in Theorem~\ref{TTT} might be (c),
because it is very tempting to believe that $\overline{\psi}(G) < 0$ if
and only if $\overline{\chi}(G^*) < 0$. However,  (a) and (c) of Theorem~\ref{TTT},
when combined with Theorem~\ref{T}(b), show that it cannot be true. This discrepancy comes from the fact
that the definition of $\overline{\chi}(G^*)$ requires some subgraphs of $G^*$ whose corresponding
subgraphs of $G$ are \emph{disconnected}. This will be explained in the subsequent sections in detail.
For Theorem~\ref{TTT}(a) and the second part (the case when vertex degrees are bounded) of Theorem~\ref{TTT}(b),
their credits should be addressed to Woess \cite{Woe}. In fact, Woess showed that
$\imath(G) > 0$ if one of the following conditions holds: $\overline{\phi}(G) < 0$, or $\overline{\psi}(G) < 0$,
or $\overline{\chi}(G) < 0$. Note that this result is already enough for the second part of Theorem~\ref{TTT}(b),
because $\imath(G)$ is quantitatively equivalent to $\jmath(G)$ when vertex degrees of $G$ are bounded.
Also one can check that Woess's arguments are enough to show (a)
only with some minor modifications. This will be explained in Section~\ref{Scur}.

It was observed independently in \cite{Hig, Zuk} that the condition $\psi(v)<0$
actually implies $\psi(v)<-\epsilon_0$ for some positive constant $\epsilon_0$. Higuchi also showed in \cite{Hig}
that one can choose $\epsilon_0 = 1/1806$.
Thus we obtain the following immediate corollary of Theorems~\ref{T} and \ref{TTT}.

\begin{corollary}\label{Cor}
Suppose $G$ is a proper planar graph. If $\psi(v)<0$ for all $v \in V(G)$,
or $\chi(f)<0$ for all $f \in F(G)$, then $\jmath(G)>0$.
Consequently, in either case we have $\imath(G)>0$, $\kappa(G)>0$, $\imath(G^*)>0$, $\jmath(G^*)>0$, and $\kappa(G^*)>0$.
\end{corollary}

Compare this corollary with \cite[Theorem~B and Corollary~2.3]{Hig} and \cite[Proposition~4]{Zuk}. Another corollary of
Theorem~\ref{TTT} is the following.

\begin{corollary}\label{sptree}
Suppose $G$ is a graph satisfying the assumptions in (a) or (b) of Theorem~\ref{TTT}. Then $G$ contains a tree $T$
such that $\jmath(T) > 0$.
\end{corollary}
\begin{proof}
This corollary is an immediate consequence of Theorem~\ref{TTT} above, and Theorem~1.1 of \cite{BS}.
For a given locally finite graph $\Gamma$ and its subgraph $S \subset \Gamma$, let $\tilde{\partial}_v S$ be
the set of all vertices in $V(\Gamma) \setminus V(S)$ that have a neighbor in $S$. Define
\[
\tilde{\jmath} (\Gamma) := \inf_S \frac{|\tilde{\partial}_v S|}{|V(S)|},
\]
where $S$ runs over all the nonempty finite subgraphs of $\Gamma$ as before. Then
Benjamini and Schramm showed in \cite{BS} that every graph $G$ with $\tilde{\jmath}(G)>0$ contains a tree $T$
with $\tilde{\jmath}(T)>0$. But one can check that
$\jmath(\Gamma) = \tilde{\jmath}(\Gamma) /(1+ \tilde{\jmath}(\Gamma))$ for every planar graph $\Gamma$,
so we have the corollary. The details are left to the readers.
\end{proof}

Our last topic is about the relation between strong isoperimetric inequalities and Gromov hyperbolicity on planar graphs. For the definition of Gromov hyperbolic spaces, see Section~\ref{SG}.

\begin{theorem}\label{TG}
Suppose $G$ is a planar graph whose face degrees are bounded.
\begin{enumerate}[(a)]
\item If $\kappa(G)>0$, then $G$ is hyperbolic in the sense of Gromov.
\item The converse of (a) is not true. That is,  $G$ could be Gromov hyperbolic (and normal),
      but $\kappa(G) = 0$.
\item The right isoperimetric constant for Gromov hyperbolicity is $\kappa(G)$.
      That is, it is possible that $G$ is (normal and) not Gromov hyperbolic, but $\imath(G)>0$.
\end{enumerate}
\end{theorem}

Theorem~\ref{TG}(a) is widely believed and even considered trivial to some experts,
but surprisingly  we could not find its
proof in the literature. Of course, however, it deserves to be written somewhere since it can save some
works like \cite[Corollary 5]{BP2} or \cite[Corollary 1]{Zuk}.
Also note that the condition $\kappa(G)>0$ is equivalent to $\jmath(G)>0$ in
the above theorem, since face degrees of $G$ are bounded.

\section{Planar graphs}\label{S2}
Let $G = \bigl( V(G), E(G) \bigr)$ be a graph,
where $V(G)$ is the vertex set and $E(G)$ is the (undirected) edge set of $G$. Every edge $e \in E(G)$
is associated with two vertices $v, w \in V(G)$,
saying that $e$ is incident to $v$ and $w$, or $e$ connects $v$ and $w$.
In this case we write $e= [v,w]$, and the vertices $v$ and $w$ are called the endpoints of $e$. Also we say that $v$ and $w$ are
neighbors of each other. A graph $G$ is called simple if there is no self-loop nor multiple edges; that is,
for every edge $[v,w]\in E(G)$ we have $v\ne w$, and for every two vertices $v, w \in V(G)$ there is
at most one edge connecting these two vertices.
A graph $G$ is called connected if it is connected as a one-dimensional simplicial complex, and
planar if there is a continuous injective map $h: G \to \mathbb{R}^2$.
The image $h(G)$ is called an embedded graph, but we will not distinguish $G$ from $h(G)$ when the embedding is fixed,
and the embedded graph will be denoted by $G$ instead of $h(G)$. We say that $G$ is embedded into
$\mathbb{R}^2$ locally finitely if every compact set in $\mathbb{R}^2$ intersects
only finite number of vertices and edges of $G$. From now on, $G$ will always be a connected simple planar graph
embedded into $\mathbb{R}^2$ locally finitely.

The closure of each component of $\mathbb{R}^2 \setminus G$
is called a face of $G$, and we denote by $F(G)$  the face set of $G$. The dual graph $G^*$ of $G$ is the
planar graph such that the vertex set of $G^*$ is just $V(G^*)=F(G)$, and for the edge set we have
$[f_1, f_2] \in E(G^*)$ if and only if $f_1$ and
$f_2$ share an edge in $G$. The degree of a vertex $v\in V(G)$, denoted by $\deg v$, is the number of neighbors of $v$,
and the degree of a face $f \in F(G)$, denoted by $\deg f$, is the number of edges in $E(G)$ surrounding $f$.
In this paper, one essential assumption about $G$ is that $G^*$ is simple, and that $3 \leq \deg v, \deg f < \infty$
for every $v \in V(G)$ and $f \in F(G)$. Under this assumption, $\deg f$
is just the degree of $f$ as a vertex in $G^*$.

A graph $S$ is called a subgraph of $G$, denoted by $S \subset G$, if $V(S) \subset V(G)$ and $E(S) \subset E(G)$.
A subgraph will be always assumed to be an induced subgraph; that is, if $v, w \in V(S)$ and $[v,w] \in E(G)$,
then $[v,w] \in E(S)$. A subgraph is called finite if $|V(S)|$ is finite, and a finite
subgraph $S \subset G$ is called simply connected if it is connected as a one-dimensional simplicial complex and
$G\setminus S$ has only one component. Here we remark that our definition of simply connectedness might
be different from the one in some other literature, since
a subgraph $S \subset G$ could be simply connected as a subset of $\mathbb{R}^2$,
but not as a subgraph of $G$ (Figure~\ref{nosimply}(a)).
However, this usually makes no difference when studying strong isoperimetric inequalities, because in the proof
we can always add to $S$ all the finite components of $G \setminus S$ (cf. proofs of Lemma~\ref{main lemma},
Theorem~\ref{T2}, etc.).
\begin{figure}[h]
\centerline{
\hfill \subfigure[a subgraph (solid polygonal line) that is not simply connected ]{\epsfig{figure=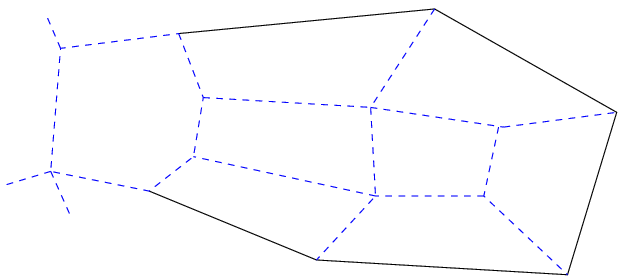, width=1.9 in, height=1.2 in}} \hfill\hfill
\subfigure[a face graph with simply connected interior but not being a polygon]{\epsfig{figure=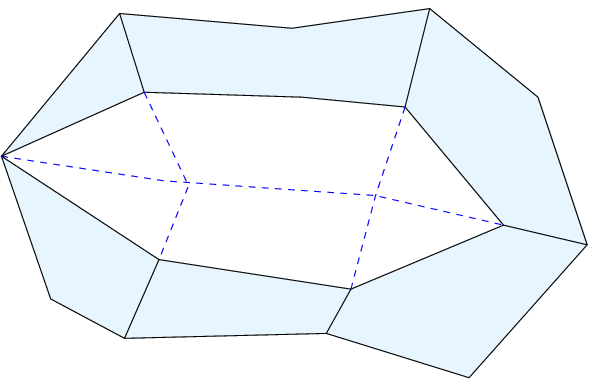, width=1.9 in, height=1.2 in}} \hfill }
\caption{a face graph $S$ (right figure) and its dual part $S^*$ (left figure)}\label{nosimply}
\end{figure}

For $S \subset G$, we define $F(S)$
as the \emph{subset} of $F(G)$  such that $f \in F(S)$ if and only if $f \in F(G)$ and it
is the closure of a component of $\mathbb{R}^2 \setminus S$. This notation is a little bit confusing, since
$F(S)$ in fact means the intersection of $F(G)$ and the face set of $S$. By abuse of the notation,
we will treat a face $f \in F(G)$ as a subgraph of $G$ consisting of all the edges and vertices on the
topological boundary of $f$. Thus we have $|E(f)|= \deg f$. Edges are also treated as subgraphs in a similar way.
A subgraph $S \subset G$ is called a \emph{face graph}
if it is a union of faces; i.e., $V(S) = \bigcup_{f \in F(S)} V(f)$. A finite subgraph is called
a \emph{polygon} if it is a simply connected face graph and the \emph{interior} of $S$,
\begin{equation}\label{E1}
D(S) := \left( \bigcup_{f \in F(S)} f \right)^\circ,
\end{equation}
is simply connected as a subset of $\mathbb{R}^2$, where
$(\cdot)^\circ$ denotes the interior of the considered set with respect to the Euclidean topology.
Note that even though a face graph has  simply connected interior, the graph itself may not be simply connected  (Figure~\ref{nosimply}(b)).  Thus in the definition we require a polygon to be simply connected.

Now suppose $S \subset G$ is a face subgraph, and define $S^*$ as the subgraph of $G^*$
such that $V(S^*)=F(S)$ (Figure~\ref{nosimply}).
Then it is easy to check that for a given face graph $S\subset G$, the interior $D(S)$ is connected if and only if
$S^*$ is connected, and such $S$ is a polygon if and only if
$S^*$ is simply connected. In particular this suggests that even though $S$ is
connected, $S^*$ may not be connected. For example, suppose $f_1, f_2$ are faces of $G$ such that
$f_1 \cap f_2$ is a vertex. Then the face graph $S\subset G$ with $F(S) =  \{ f_1, f_2 \}$ is simply connected,
but $S^*$ is definitely disconnected. Thus if we define
\begin{equation}\label{chibar}
\overline{\chi}_1 (G) = \limsup_{|F(S)|\to \infty} \overline{\chi}(F(S)),
\end{equation}
where at this time the limit superior is taken over all the \emph{polygons} $S \subset G$, then we have
$\overline{\psi}(G^*) =\overline{\chi}_1 (G) \leq \overline{\chi}(G)$. Therefore the condition $\overline{\chi}(G) < 0$
is stronger than the condition $\overline{\psi}(G^*)<0$, and this observation has delivered us to Theorem~\ref{TTT}.

For a face subgraph of $S \subset G$, there is a bijection
between $\partial S^*$ and $\partial_e S$. However, there is no set defined in $S$ which corresponds to $\partial_v S^*$.
Thus if $S$ is a subgraph of $G$, which may or may not be a face graph,
let us define the \emph{face boundary} of $S$ by
\[
\partial_f S := \{ f \in F(S) : E(f) \cap \partial_e S \ne \emptyset \}.
\]
That is, $\partial_f S$ is the set of faces in $F(S)$, each of
which shares an edge with a face in $F(G) \setminus F(S)$.

We finish this section with the following lemma. Recall that a planar graph $G$ is called proper
if every face of $G$ is a topological closed disk, and normal if it is proper and any two different faces of $G$
share at most a vertex or an edge.

\begin{lemma}\label{>2}
Suppose $G$ is a proper planar graph.
Then for every finite and connected subgraph $S \subset G$ with $|V(S)| \geq 2$, we have $|\partial_v S| \geq 2$.
Moreover, if $G$ is normal and $S \subset G$ has the property $|V(S)| \geq 3$, then $|\partial_v S| \geq 3$.
\end{lemma}
\begin{proof}
Suppose $S \subset G$ is a finite and connected subgraph of $G$. Then
we cannot have $\partial_v S = \emptyset$ since $G$ is connected.
If $|\partial_v S| = 1$,  there must be a unique face $f \in F(G)$ surrounding $S$; i.e.,
there is only one face $f \in F(G) \setminus F(S)$ such that $E(f) \cap E(S) \ne \emptyset$.
Then $f$ cannot be a topological closed disk because $S$ contains an edge, contradicting the
properness condition of $G$.

Now suppose $G$ is normal, $S \subset G$, $\partial_v S =  \{ v_1, v_2 \}$, and $|V(S)| \geq 3$.
In this case there are exactly two faces, say $f_1$ and $f_2$, whose union surrounds $S$.
Then we must have $v_1, v_2 \in V(f_1) \cap V(f_2)$, so the normality condition of $G$ implies
that $f_1 \cap f_2 = [v_1, v_2]$. This means that $S$ is just an edge because both $f_1$ and $f_2$ are
topological closed disks, which is definitely a contradiction since
 $|V(S)|\geq 3$.
\end{proof}

\section{Proof of Theorem~\ref{T}(a)}\label{a}
First,  let us derive a useful formula which will be used frequently.
Let $S$ be a finite simply connected subgraph of $G$.
Then every vertex $v \in V(S) \setminus \partial_v S$ corresponds to at least three edges in $E(S)$
and each vertex $v \in \partial_v S$ corresponds to at least one edge in $E(S)$,
while every edge in $E(S)$ corresponds to exactly two vertices in $V(S)$. Therefore we have
\[
2 |E(S)| \geq 3(|V(S)| - |\partial_v S|) +  |\partial_v S| = 3|V(S)| - 2|\partial_v S|.
\]
Thus by the Euler's formula we obtain
\[
|E(S)| +1 = |V(S)| + |F(S)| \leq \frac{2}{3} \cdot |E(S)| + \frac{2}{3} \cdot |\partial_v S| + |F(S)|,
\]
or
\begin{equation}\label{facebound}
|E(S)| \leq 2 |\partial_v S| + 3|F(S)| - 3.
\end{equation}

\bigskip

\begin{lemma}\label{main lemma}
For a given planar graph $G$, the following two conditions are equivalent:
\begin{enumerate}[(a)]
\item for every finite $S \subset G$, there is a constant $C_1$
      such that $|F(S)| \leq C_1 |\partial_e S|$;
\item for every finite $S \subset G$, there is a constant $C_2$
      such that $\sum_{f \in F(S)} \deg (f)  \leq C_2 |\partial_e S|$.
\end{enumerate}
Furthermore, the constants $C_i$ are quantitative to each other.
\end{lemma}
\begin{proof}
We only need to show the implication (a) $\to$ (b), since the other direction is trivial. In this case,
we can assume without loss of generality that $S$ is a  polygon, since otherwise we can
add to $S$ all the finite components of $G \setminus S$ and
cut from $S$ all the edges not surrounding a face in $F(S)$.  Note that these operations make
either the set $\partial_e S$ smaller, or the set $F(S)$ bigger, or both.
If $D(S)$ consists of several components, we can consider them separately.

Then since $S$ is a polygon we have $|\partial_v S| \leq |\partial_e S|$ and
\begin{equation}\label{edgenumber}
2 |E(S)| = \sum_{f \in F(S)} \deg (f) + |\partial_e S|.
\end{equation}
Now equations  \eqref{facebound} and \eqref{edgenumber} imply that
\[
 \sum_{f \in F(S)} \deg (f) = 2|E(S)| - |\partial_e S| \leq 4|\partial_v S| + 6|F(S)| - |\partial_e S|
 \leq 3 |\partial_e S| + 6|F(S)|,
\]
and we obtain the implication (a) $\to$ (b) with $C_2 = 6C_1 +3$. In fact one can check,
by modifying \eqref{facebound}, that
a better estimate $C_2 \leq 6 C_1 +1$ holds since $S$ is a polygon hence every vertex in $\partial_v S$ corresponds
at least two edges in $E(S)$. This completes the proof.
\end{proof}

Note that the condition (a) of the previous lemma is equivalent to the condition $\kappa(G) >0$
and,  considering the duality property, one can check that (b) is equivalent to $\imath(G^*)>0$.
Thus Lemma~\ref{main lemma} implies the equivalence between $\imath(G^*)>0$ and $\kappa(G)>0$.
The equivalence between $\imath(G)>0$ and $\kappa(G^*)>0$ comes from  the duality.
This completes the proof of Theorem~\ref{T}(a).

\section{Partitioning a finite  subgraph}\label{main section}
In this section we will prove Theorem~\ref{T2} and the rest of Theorem~\ref{T}.
Our main strategy is partitioning a given finite subgraph into `nice' subgraphs, and passing
strong isoperimetric inequalities of `nice' subgraphs to the given one.

In the lemma below, the notation $S_1 \cup S_2$, where $S_1, S_2$ are subgraphs of $G$, means
the subgraph $S \subset G$ with $V(S) = V(S_1) \cup V(S_2)$ and $E(S)=E(S_1)\cup E(S_2)$.
A priori, therefore, $S_1 \cup S_2$ does not have to be an \emph{induced} subgraph,
but we will only consider the case when it is  induced.
The graph $S_1 \cap S_2$ is defined similarly.

\begin{lemma}\label{partition}
Suppose $S$ is a finite subgraph of $G$. If $S = S_1 \cup S_2 \cup \cdots \cup S_n$,
where $S_1, S_2, \ldots, S_n$ are subgraphs of $G$ such that
\begin{enumerate}[(a)]
\item $|V(S_i)| \leq C |\partial_v S_i|$ for each $i =1,2, \ldots, n$,
\item $S_1 \cup S_2 \cup \cdots \cup S_{i}$ is an induced subgraph for each $i =1,2, \ldots, n$,
\item $|V\bigl( (S_1 \cup S_2 \cup \cdots \cup S_{i-1} ) \cap S_i\bigr) | = 1$ for each
      $i =2,3, \ldots, n$, and
\item $|\partial_v S| \geq n / \tau$ for some $\tau > 0$,
\end{enumerate}
then we have $|V(S)|  \leq (1+2\tau)C|\partial_v S|$.
\end{lemma}
\begin{proof}
From (b) and (c), it is easy to see that
\begin{equation}\label{22}
|\partial_v (S_1 \cup \cdots \cup S_{i})| \geq |\partial_v (S_1 \cup \cdots \cup S_{i-1})| + |\partial_v S_i| -2.
\end{equation}
Therefore,
\begin{align*}
|V(S)| & = |V(S_1 \cup \cdots \cup S_n)| \leq |V(S_1)| + |V(S_2)| + \cdots + |V(S_n)|\\
       & \leq C (|\partial_v S_1| + |\partial_v S_2| + \cdots + |\partial_v S_n|)\\
       & \leq C( |\partial_v (S_1 \cup S_2)| + |\partial_v S_3| + \cdots + |\partial_v S_n| +2)  \\
       & \leq \cdots  \leq C(|\partial_v (S_1 \cup \cdots \cup S_n)| + 2n-2)
        \leq (1+2 \tau) C |\partial_v S|,
\end{align*}
as desired.
\end{proof}

Suppose $S$ is a finite simply connected subgraph of $G$. If $T \subset S$ is a polygon such that
no $T'$ with $T \subsetneq T' \subset S$ is a polygon, we call $T$ a \emph{leaf} of $S$.
If $T \subset S$ is finite tree such that $E(T) \cap E(f) = \emptyset$ for every $f \in F(S)$, and
 no $T'$ with $T \subsetneq T' \subset S$ is a tree satisfying the same property,
we call such $T$ a \emph{branch} of $S$. If $T \subset S$ is either a leaf or a branch,
we call $T$ a \emph{part} of $S$.

\begin{proof}[Proof of Theorem~\ref{T2}: the normal case]
Let $G$ be a \emph{normal} planar graph satisfying the assumptions in Theorem~\ref{T2}, and suppose
a finite subgraph $S \subset G$ is given. To prove Theorem~\ref{T2}, we may assume
that $S$ is simply connected. Otherwise we can add to $S$ all the finite components of $G \setminus S$ and consider
each component of $S$ separately.

Choose an edge $e_1 \in E(S)$, and note that there exists a unique leaf or branch $S_1$ of $S$ such that $e \in E(S_1)$,
depending on whether $e \in E(f)$ for some $f \in F(S)$ or not, respectively.
If $S_1 \ne S$, choose $e_2 \in E(S) \setminus E(S_1)$ with only one end in $V(S_1)$, and let $S_2$ be
the part of $S$ such that $e_2 \in E(S_2)$. Then because $S_1$ and $S_2$ are \emph{maximal} polygons or trees
of the \emph{simply connected} subgraph $S$, one can see that
$S_1 \cap S_2$ is a single vertex that is an end of $e_2$, and $S_1 \cup S_2$ is an induced subgraph of $S$.
If $S_1 \cup S_2 \ne S$,  choose $e_3 \in E(S) \setminus E(S_1 \cup S_2)$ with only one end in
$V(S_1 \cup S_2)$, and repeat the same process.
This process cannot run forever, since $S$ is a finite graph. Thus we just have enumerated the
parts $S_1, S_2, \ldots, S_n$ of $S$ so that
$S = S_1 \cup S_2 \cup \cdots \cup S_n$ and they satisfy the conditions (b) and (c) of Lemma~\ref{partition}.

Each $S_i$ is either a branch or a leaf, i.e., a finite tree or a polygon. If $S_i$ is a leaf,
the inequality $|V(S_i)| \leq C|\partial_v S_i|$ holds by our assumption. If $S_i$ is a branch, then
since $F(S_i) = \emptyset$ and $|V(S_i)| = |E(S_i)| +1$, we have $|V(S_i)| \leq 2 |\partial_v S_i|$ by \eqref{facebound}.
Thus the condition (a) of Lemma~\ref{partition} is also satisfied for the sequence $S_1, \ldots, S_n$.

It remains to show that $|\partial_v S| \geq n/3$. To see this, let $k$ be the number of $S_i$
such that $|\partial_v S_i| = 2$ and $\partial_v S \cap V(S_i) \ne \emptyset$. If $k \geq n/3$,
there is nothing more to prove since such $S_i$ must be an edge by Lemma~\ref{>2}, hence
such $S_i$'s are disjoint by our definition of a branch.
Now suppose that $k < n/3$. Then a part $S_i$ satisfying $|\partial_v S_i| = 2$ and
$\partial_v S \cap V(S_i) = \emptyset$ must be an edge whose both ends are connected to leaves. Thus
the number of such $S_i$'s cannot exceed $(n-k)/2$. This means that there are at least $(n-k)/2$ parts with
$|\partial_v S_i| \geq 3$. Since such $S_i$ makes the left hand side of  \eqref{22} increase at least by one, we conclude that
\[
|\partial_v S| = |\partial_v (S_1 \cup \cdots \cup S_n)| \geq \frac{n-k}{2} \geq \frac{n}{3}
\]
since $k < n/3$.
This completes the proof.
\end{proof}

\begin{figure}[hbt]
\begin{center}
 \input{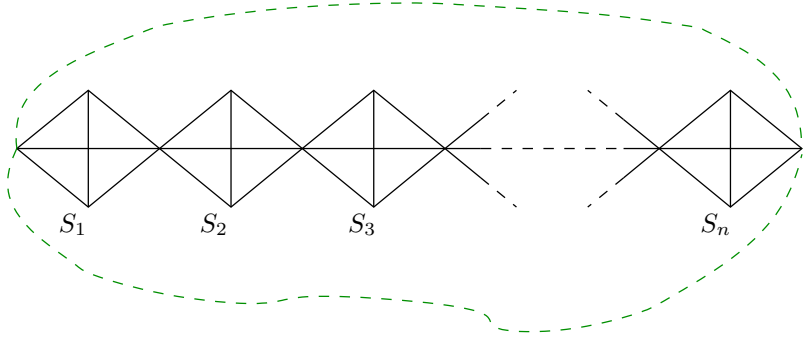}
 \caption{a problem when partitioning into polygons}\label{problem}
\end{center}
\end{figure}
The approach used in the previous proof, that is, partitioning $S$ into parts and using Lemma~\ref{partition},
has some problems when $G$ is not normal.
Let $\Lambda$ be the graph of Figure~\ref{nonproper}(a) given in the introduction,
and for each $i =1,2,\ldots,n$, let $S_i$ be a copy of $\Lambda$. We next connect them
back to back as in  Figure~\ref{problem} and obtain a new graph
$S := S_1 \cup S_2 \cup \cdots \cup S_n$. Also suppose $S$ is enclosed
by the union of two faces with huge face degrees. Then even though
each $S_i$ satisfies the inequality $|V(S_i)| \leq 3 |\partial_v S_i|$, we cannot pass this
inequality to $S$; i.e., we have  $|\partial_v S|/|V(S)| \to 0$ as  $n \to \infty$.
As we saw in the previous proof, however, this problem cannot occur  when $G$ is normal.
The next lemma says that it cannot happen either if $\jmath(G^*)>0$.

\begin{lemma}\label{chain}
Suppose $\jmath(G^*)>0$, and let $S$ be a finite simply connected subgraph of $G$ satisfying the following two properties:
\begin{enumerate}[(a)]
\item every branch of $S$ is a path; i.e., a finite union of consecutive edges without self-intersections;
\item $P \cap P' \cap P'' = \emptyset$ for every three distinct parts $P, P', P''$ of $S$.
\end{enumerate}
Then there is an absolute constant $C$ such that $|V(S)|  \leq C |\partial_v S|$.
\end{lemma}
Note that the subgraph $S$ in Lemma~\ref{chain} has the following property: for any vertex $v \in V(S)$ and
every sufficiently small neighborhood $U \ni v$ in $\mathbb{R}^2$,
$U \setminus (S \cup D(S))$ has at most two components.
\begin{proof}
Let $\{ v_1, v_2, \ldots, v_m \}$ be an enumeration of $\partial_v S$ along the boundary of $S$.
In other words, when walking along the topological boundary of $S \cup D(S)$ either clockwise or
counterclockwise, we enumerate the vertices in $\partial_v S$ in the order we encounter them,
allowing some multiple counts. However, any
vertex in $\partial_v S$ will not be counted more than twice by the assumptions (a) and (b). Consequently
we have $2|\partial_v S| \geq m$.

Note that $S$ is a simply connected graph with no outbound edge in $E(G) \setminus E(S)$
between $v_i$ and $v_{i+1}$, where $i$ is in mod $m$.
Thus for each $i \in \{1, 2, \ldots, m\}$, there is the face $f_i$ attached to $S$ between $v_i$ and $v_{i+1}$.
To be precise, $f_i$ is the face in $F(G) \setminus F(S)$
such that $E(f_i)$ contains all the edges in $\partial_e S$ between $v_i$ and $v_{i+1}$ (Figure~\ref{facegraph}).

\begin{figure}[tb]
\begin{center}
 \input{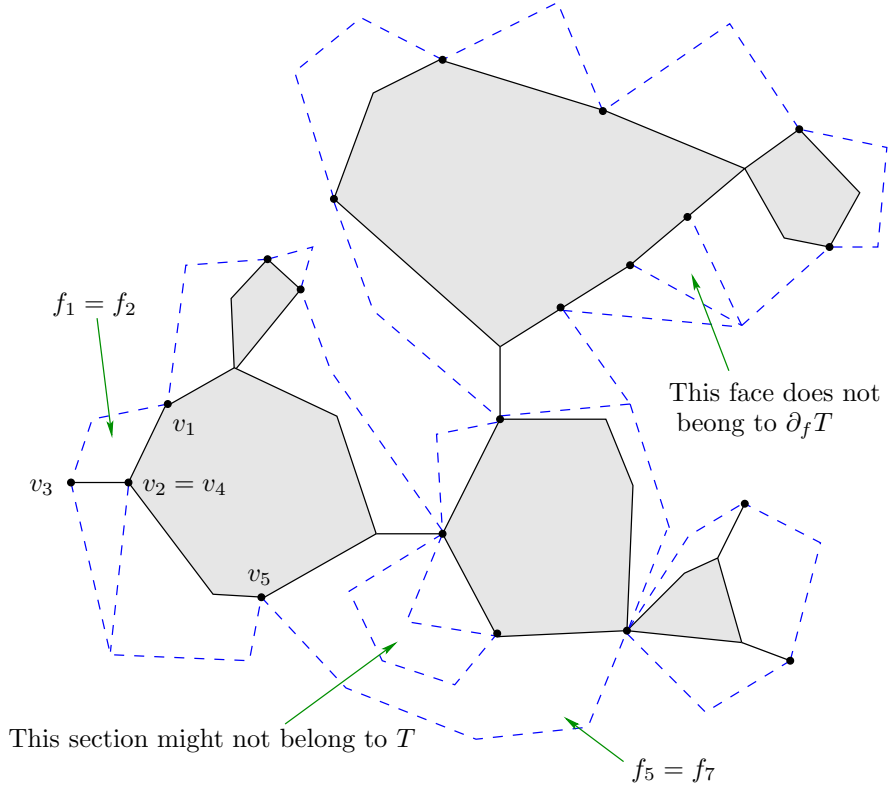}
 \caption{a sketch of a subgraph satisfying the assumptions in Lemma~\ref{chain}; this graph consists of 6 polygons and 5 edges, and dotted lines indicate the faces attached to $S$; the section that does not belong to $T$
 would be some union of faces in $F(G) \setminus F(T)$}\label{facegraph}
\end{center}
\end{figure}

Let $T$ be the subgraph of $G$ such that $V(T) = V(S) \cup \bigcup_{i=1}^m V(f_i)$.
Definitely $T$ is a face graph with $\partial_f T \subset \{ f_1, f_2, \ldots, f_m \}$, hence
our assumption $\jmath(G^*)>0$ implies that $|V(T^*)| \leq C_1 |\partial_v T^*|$ for $C_1 = \jmath(G^*)^{-1}$, or
\[
|F(T)| \leq C_1 |\partial_f T| \leq C_1 \cdot m \leq 2C_1 |\partial_v S|.
\]
 Thus by \eqref{facebound} we have
\begin{align*}
|V(S)| & = |E(S)| - |F(S)| + 1 \leq 2 |\partial_v S| + 2|F(S)| \\
& \leq 2 |\partial_v S| + 2|F(T)| \leq (2+ 4C_1) |\partial_v S|,
\end{align*}
as desired.
\end{proof}

\begin{corollary}\label{c2}
Suppose $G$ is a proper planar graph with bounded face degrees such that $|V(T)|\leq C|\partial_v T|$ for
every polygon $T \subset G$. If $S$ is a simply connected subgraph of $G$  satisfying the conditions (a) and (b)
in Lemma~\ref{chain}, then $|V(S)| \leq C_1 |\partial_v S|$ for some constant $C_1$ not depending on $S$.
\end{corollary}
\begin{proof}
We follow the proof of Theorem~\ref{T2}, the normal case, and enumerate the parts $S_1, S_2, \ldots, S_n$ of $S$
so that they satisfy the conditions (a), (b), and (c) of Lemma~\ref{partition}. Then as in Lemma~\ref{chain},
we enumerate $\partial_v S$ by $v_1, v_2, \ldots, v_m$, where some vertices are counted twice, and
let $f_1, f_2, \ldots, f_m$ be the faces attached to $S$ between $v_i$ and $v_{i+1}$, where $i$ is in mod $m$.

By the definition of a part, we have $\partial_e S \cap E(S_i) \ne \emptyset$ for each $i=1,2,\ldots, n$, so
we must have $|\partial_e S| \geq n$. On the other hand,  face degrees of $G$ are bounded above, say by $M$,
so the number of edges in $E(f_i) \cap \partial_e S$ cannot exceed $M$. Thus the inequality $m M \geq |\partial_e S|$ holds.
Since $m \leq 2|\partial_v S|$,  we have $|\partial_v S| \geq n/(2M)$.
The corollary follows from Lemma~\ref{partition}.
\end{proof}

The problem in Figure~\ref{problem} is caused by the fact that face degrees of $G$ are not bounded,
and we have seen so far that this phenomenon does not occur under the assumptions of either
Theorem~\ref{T} or Theorem~\ref{T2}. However, there is still a problem in partitioning a subgraph
into leaves and branches, since we still have to worry about the unboundedness of \emph{vertex} degrees. For example
the proof in Lemma~\ref{chain} does not work for a general subgraph $S \subset G$
since when one walks along the topological boundary of $S \cup D(S)$,
there may be a vertex in $\partial_v S$ which appears too many times. So to overcome this obstacle,
one may need a different kind of partition.

\begin{lemma}\label{key}
Suppose $G$ is a proper planar graph and $S$ is a finite simply connected subgraph of $G$. Then there exists a partition
$S_1 \cup S_2 \cup \cdots \cup S_n = S$ which satisfies the conditions (b), (c), and (d) of Lemma~\ref{partition}.
Furthermore, we can make the partition so that each $S_i$ satisfies the conditions (a) and (b) of Lemma~\ref{chain}.
\end{lemma}
\begin{proof}
Suppose $S \subset G$ is given, and let $V_1 := \{L_1, L_2, \ldots, L_k \}$ be the set of all the leaves of $S$.
Define $V^i$, $i=1,2, \ldots, k$, such that
\[
V^i := \{ v \in \partial_v L_i : \mbox{there exists an edge } e \in E(S) \setminus E(L_i) \mbox{ with one end at } v \}.
\]
\begin{figure}[tb]
\begin{center}
 \includegraphics{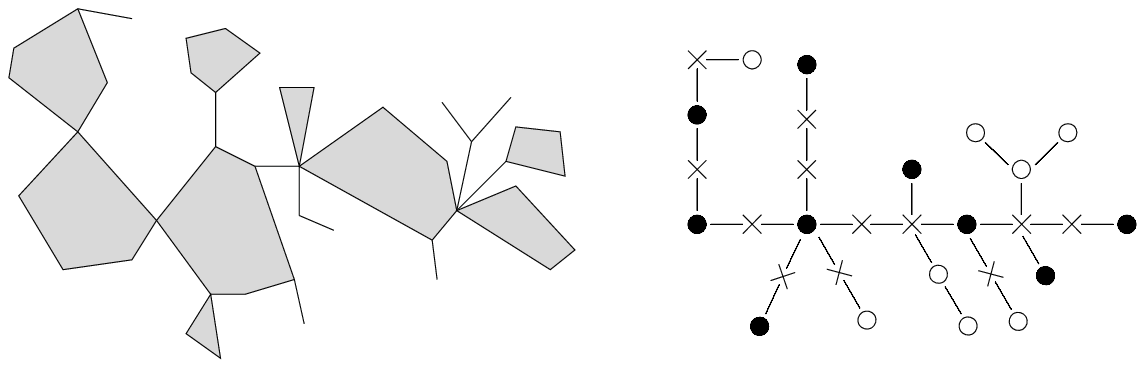}
 \caption{a sketch of a simply connected graph $S$ (left) and the tree $T$ obtained from $S$ (right); on the tree $T$,
  the symbols $\CIRCLE, \Circle$, and $\times$ indicate the vertices in $V_1$, $V_2$, and $V_3$, respectively.}\label{partition1}
\end{center}
\end{figure}
Then by reducing each $L_i$ to a vertex, connecting it to the vertices in $V^i$,
and keeping the rest of $S$ unchanged, we get a new finite tree $T$ (Figure~\ref{partition1}).
The reason we add to $T$ the vertices in $V^i$ is because we do not want to change the combinatorial pattern
of $S$. Formally, $T$ is the graph with $V(T) := V_1 \cup V_2 \cup V_3$, where $V_1$ is as above,
$V_2 := V(S) \setminus V(L_1 \cup L_2 \cup \cdots \cup L_k)$, and $V_3 := \bigcup_{i=1}^k V^i$, and
we define the edge set $E(T)$ so that
$[v, w] \in E(T)$ if one of the following holds: (1) $v, w \in V_2 \cup V_3$ and $[v, w] \in E(S)$, or (2)
$v \in V_1, w \in V_3$, and $w \in V(v)$, or (3) $v \in V_3, w \in V_1$, and $v \in V(w)$.
Such $T$ must be a tree since $S$ is simply connected.

In $T$, let $A = \{ v \in V(T) : \deg_T v \geq 3 \}$. Here $\deg_T v$ denotes the number of edges in $T$
with one end at $v$. Then we consider $T$ a simplicial complex, and let $\{ T_1, T_2, \ldots, T_m \}$ be
the set of the closures of  each components of $T \setminus A$ (Figure~\ref{partition2}).
Note that each $T_j$ is isomorphic to a finite path.
Moreover for $i \ne j$, $T_i \cap T_j$ is either empty or a vertex in $A$, hence we can enumerate $\{T_j\}$
so that $T_1 \cup T_2 \cup \cdots \cup T_{j}$ is connected for  $j =1,2, \ldots, m$ and
$(T_1 \cup T_2 \cup \cdots \cup T_{j-1}) \cap T_j$ is a single vertex for  $j =2,3, \ldots, m$.
\begin{figure}[tb]
\begin{center}
 \input{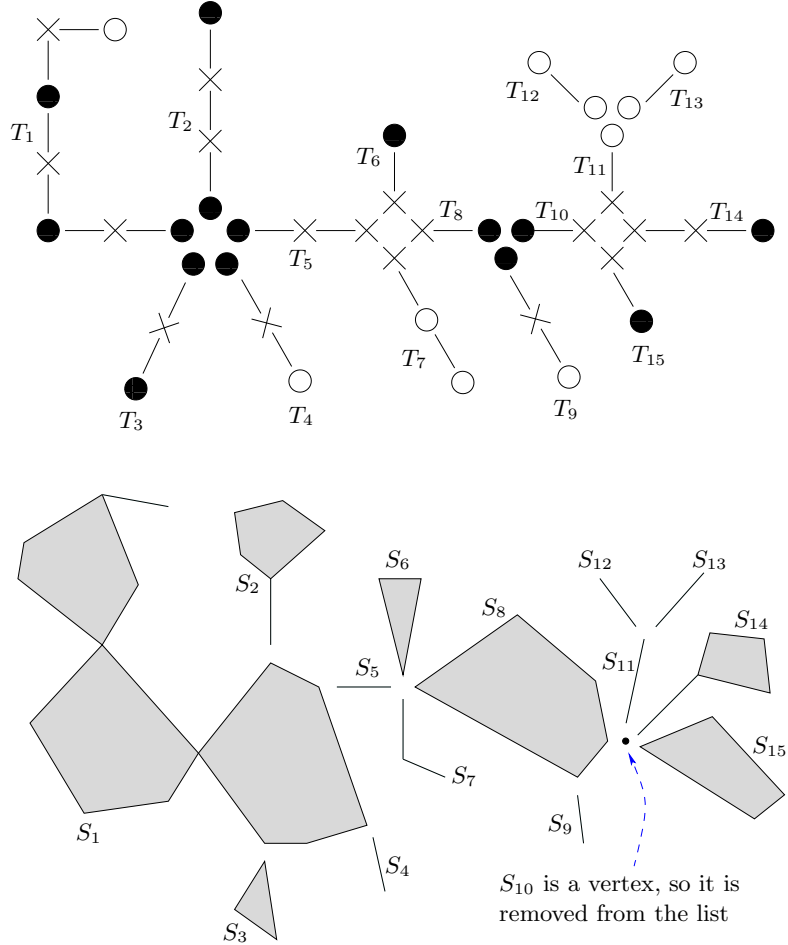}
 \caption{$T$ is partitioned so that each $T_j$ is a path (top), and $S$ is partitioned
 associated with the partition of $T$ (bottom).}\label{partition2}
\end{center}
\end{figure}

For each $j$ we assign a subgraph $S_j \subset S$ in an obvious way, but some modification is needed.
For $S_1$, we define it so that
\[
V(S_1) = \Bigl( V(T_1) \cap (V_2 \cup V_3) \Bigr) \cup \bigcup_{v \in V(T_1) \cap V_1} V(v),
\]
and for $j =2,3, \ldots, m$ we define $S_j$ so that
\[
V(S_j) = \Bigl( V(T_j) \cap (V_2 \cup V_3) \Bigr) \cup
\bigcup_{v \in V(T_j \setminus (T_1 \cup \dots \cup T_{j-1})) \cap V_1} V(v).
\]
In other words, $S_j$ is the graph consisting of the vertices in $T_j \cap (V_2 \cup V_3)$ and the leaves
in $T_j \cap V_1$, but from $S_j$ we remove the leaf in $T_1 \cup \cdots \cup T_{j-1}$, if any.
But if a leaf is removed from $S_j$,  it must correspond to a vertex $v = L_i \in V(T)$ such that
$\{ L_i \} =(T_1 \cup T_2 \cup \cdots \cup T_{j-1}) \cap T_j$. Then since $T_j$ contains a vertex
in $V^i$, $(S_1 \cup \cdots \cup S_{j-1}) \cap S_j$ must be a single vertex.
Clearly $(S_1 \cup \cdots \cup S_{j-1}) \cap S_j$ is also a single vertex when
no leaf is removed from $S_j$.

Some of $S_j$ could be just a vertex in $V_3$, so we eliminate all such $S_j$'s from the list.
Now we have just obtained a sequence of subgraphs $\{ S_{n_1}, S_{n_2}, \ldots, S_{n_s} \}$
which satisfies the condition (c) of Lemma~\ref{partition}. The condition (b) of Lemma~\ref{partition}
easily follows from the construction. Moreover, since each $T_{n_i}$ is a path,
each $S_{n_i}$  satisfies the conditions (a) and (b) of Lemma~\ref{chain}.

Now it remains to verify the condition (d) of Lemma~\ref{partition}.
Let $B := \{ v \in V(T) : \deg_T v =1\}$. If $v \in B$, then $v \notin V_3$ by the
definition of $V^i$. Thus $v \in V_1 \cup V_2$.
If $v \in V_2$, then $v$ is in fact a vertex in $\partial_v S$. If $v \in V_1$, it corresponds
to a leaf $L_i$ with $|V^i| =1$. Then since $|\partial_v L_i| \geq 2$ by Lemma~\ref{>2}, we can assign $v$
to a vertex $w \in \partial_v S \setminus V_3$. Thus there exists an injection map from
$B$ into $\partial_v S$, so the inequality $|B| \leq |\partial_v S|$ holds.

Now in $T$, we replace each $T_j$ by an edge and get another tree $T'$.
Then the number of edges in $T'$ is exactly $m$,
and every vertex in $T'$ has degree either 1 or $\geq 3$. Also the number of vertices $v \in V(T')$ with
$\deg_{T'} v =1$ is
exactly $|B|$. Therefore by a computation similar to \eqref{facebound} we obtain
\[
m = |E(T')| \leq 2|B| \leq 2 |\partial_v S|.
\]
Since $s \leq m$, we conclude that the sequence $\{ S_{n_1}, S_{n_2}, \ldots, S_{n_s} \}$ also satisfies
the condition (d) of Lemma~\ref{partition}. Since $S = S_{n_1} \cup S_{n_2} \cup \cdots \cup S_{n_s}$,
this finishes the proof.
\end{proof}

\begin{proof}[Proofs of Theorems~\ref{T} and \ref{T2}.]
Theorem~\ref{T}(a) was already proved. For (b) of Theorem~\ref{T}, the implication $\jmath(G^*)>0 \to \jmath(G)>0$
is a consequence of Lemmas~\ref{partition}, \ref{chain}, and \ref{key}, and the converse comes from the duality.
Theorem~\ref{T}(c) is an easy consequence of Theorem~\ref{T}(a) and (b),
so we leave the details to the reader.

The normal case of Theorem~\ref{T2} was already proved, and the case when face degrees of $G$ are bounded
comes from Lemmas~\ref{partition}, \ref{key}, and Corollary~\ref{c2}.
This completes the proofs of Theorems~\ref{T} and \ref{T2}.
\end{proof}

\section{Combinatorial curvatures and strong isoperimetric inequalities}\label{Scur}
In this section we deal with combinatorial curvatures and prove Theorem~\ref{TTT}.
But as explained in the introduction, Theorem~\ref{TTT}(a) was essentially proved in \cite{Woe}.
To see it, and to prove Theorem~\ref{TTT}(b),
suppose $S$ is a subgraph of a proper graph $G$. As before, we can assume that
$S$ is simply connected without loss of generality.

When $\overline{\phi}(G) < 0$, Woess showed the inequality $|E(S)| \leq C |\partial_v S|$, where $C$ is
an absolute constant, in the proof of Theorem~1 of \cite{Woe}. But since $|V(S)| \leq |E(S)|+1 \leq 2|E(S)|$
by the Euler's formula, the conclusion $\jmath(G)>0$ follows immediately. When $\overline{\chi}(G) < 0$,
the inequality $|F(S)| \leq C |\partial_v S|$ was obtained in the proof of Theorem~2(b) of \cite{Woe},
where $C$ is an absolute constant. Then by \eqref{facebound} and the Euler's formula we have
\begin{equation}\label{cure}
|V(S)| = |E(S)|-|F(S)|+1 \leq 2|\partial_v S| + 2|F(S)| - 2 \leq (2+2C)|\partial_v S|,
\end{equation}
showing that $\jmath(G)>0$.

We next consider Theorem~\ref{TTT}(b).
When vertex degrees of $G$ are bounded, Theorem~\ref{TTT}(b) follows from Theorem~2(a) of \cite{Woe}
as explained in the introduction, but let us prove it in a different way.
Since $\overline{\psi}(G) = \overline{\chi}_1(G^*)$, where $\overline{\chi}_1(G)$ is defined in \eqref{chibar},
the assumption $\overline{\psi}(G) < 0$ implies the inequality $\overline{\chi}_1(G^*) <0$.
Then we follow the proof of Theorem~2(b) of \cite{Woe} and use \eqref{cure}, and confirm that
there exists an absolute constant $C$ such that $|S| \leq C |\partial_v S|$ for every \emph{polygon}
$S \subset G^*$.
Thus we have $\jmath(G^*) >0$ by Theorem~\ref{T2} if either $G^*$ is normal or face degrees of $G^*$
are bounded. Since $G$ is normal if and only if $G^*$ is normal, and face degrees of $G^*$ are bounded
if and only if vertex degrees of $G$ are bounded, Theorem~\ref{TTT}(b) follows from Theorem~\ref{T}(b).

For Theorem~\ref{TTT}(c), we construct a planar graph $G$ such that $\overline{\psi}(G) < 0$
but $\jmath(G) =0$. Note that such graph cannot be normal.

For $n \in \mathbb{Z}$, let $S_n$ be the finite graph such that
\[
V(S_n) = \{ o_n, o_{n+1}, b_1^n, b_2^n, v_1^n, v_2^n, \ldots, v_{2|n|-1}^n, v_{2|n|}^n \}
\]
with edges $[o_{n+i}, b_j^n], [o_{n+i}, v_{2k-1}^n], [o_{n+i}, v_{2k}^n],$ and $[v_{2k-1}^n, v_{2k}^n]$,
where $0 \leq i \leq 1$, $1 \leq j \leq 2$, and $1 \leq k \leq |n|$. Then
for each $n \in \mathbb{Z}$, $S_n$ and $S_{n+1}$ share the vertex $o_{n+1}$, so
$S := \bigcup_{n=-\infty}^\infty S_n$ is a connected planar graph.
Furthermore, we can embed it into $\mathbb{R}^2$
in such a way that each $v_k^n$ is enclosed by the cycle $[o_n, b_1^n] \cup [b_1^n, o_{n+1}] \cup
[o_{n+1}, b_2^n] \cup [b_2^n, o_{n}]$. See Figure~\ref{nonnormal}.
\begin{figure}[h]
\begin{center}
 \input{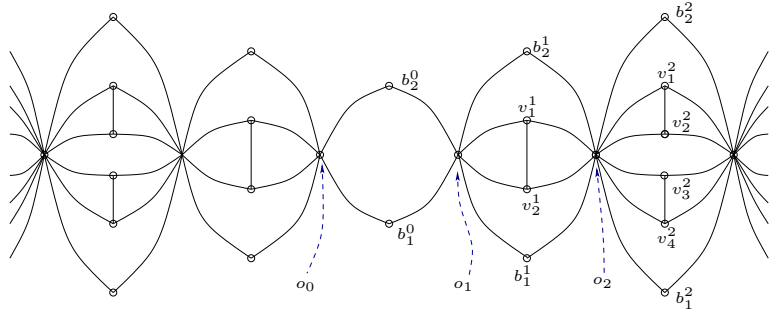}
 \caption{the graph $S$}\label{nonnormal}
\end{center}
\end{figure}

On the unbounded faces of $S$, we add vertices and edges so that the resulting graph $G$ satisfies
the following properties: $G$ is a simple proper planar graph with $3 \leq \deg v, \deg f < \infty$ for every
$v \in V(G)$ and $f \in F(G)$,
$\deg v \geq 7$ for every added vertex $v \in V(G) \setminus V(S)$, $\deg b_j^n \geq 7$ for
$j=1,2$ and $n \in \mathbb{Z}$, and $\deg o_n \geq 14 |n| + 14$ for all $n \in \mathbb{Z}$.
In other words, we for example triangulate both of the unbounded faces of $S$ so that every added vertices
and $b_j^n$'s are of degree at least 7 and the degrees of $o_n$ are huge enough. Note that
this can be done by mathematical induction.

It is easy to see that $\kappa(G)=\jmath(G)=0$ since $|\partial_v S_n| = |\partial_e S_n|=4$
but $|F(S_n)| = 3|n| +1$ and $|V(S_n)|=2|n| +4$. On the other hand,
we have $\psi(v) \leq -1/6$ if $v \in V(G) \setminus V(S)$ or $v = b_j^n$ for some $j=1,2$ and $n \in \mathbb{Z}$.
Also direct computation shows that
\begin{align*}
\psi(v_k^n) &= 1 - \frac{3}{2} + \frac{2}{3} + \frac{1}{4} = \frac{5}{12}  \quad \mbox{for all } n \in \mathbb{Z}
       \mbox{ and } 1 \leq k \leq 2|n|;\\
\psi(o_n)  &\leq 1 - \frac{\deg o_n}{2} + \frac{\deg o_n}{3} \leq - \frac{ 7 |n| +4}{3} \quad \mbox{for all } n \in \mathbb{Z}.
\end{align*}
Thus the vertex  curvature function $\psi$ assumes a positive value only at $v_k^n$, but
they are dominated by the negative vertex curvature at $o_n$. In fact, for each $o_n$ there are
at most $4|n|+2$ neighboring vertices $v_k^n$ and $v_k^{n-1}$, so
\[
\psi (o_n) + \sum_{k=1}^{2|n|} \psi(v_k^n) + \sum_{k=1}^{2|n-1|} \psi(v_k^{n-1})
\leq - \frac{ 7 |n| +4}{3} + \frac{5(4|n|+2)}{12} = -\frac{4|n| +3}{6}.
\]
This means that if $T$ is a finite connected subgraph of $G$ with $|V(T)| \geq 3$, then
$\sum_{v \in V(T)} \psi(v) \leq -|V(T)|/6$. Consequently we have  $\overline{\psi}(G) \leq -1/6 < 0$,
and this completes the proof of Theorem~\ref{TTT}. Note that the face graph $U_n \subset G^*$
with $F(U_n)= V(U_n^*)=\{ v_k^n : 1 \leq k \leq 2|n| \}$ is a connected graph which looks similar to the one
in Figure~\ref{problem}, while $U_n^*$ is a disconnected graph in $G$ (Figure~\ref{dual}).
\begin{figure}[hbt]
\begin{center}
 \input{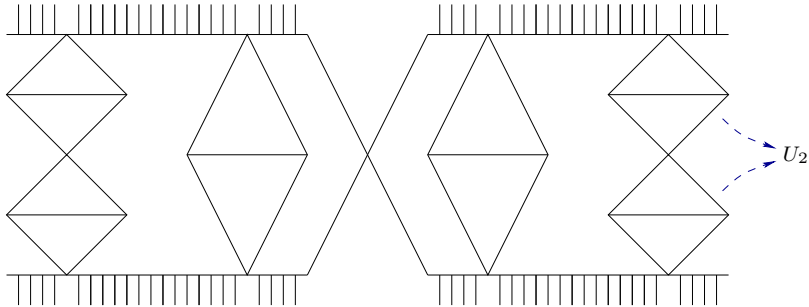}
 \caption{the dual graph of $G$; the subgraph $U_2$ corresponds to the disconnected
  graph $U_2^* = [v_1^2, v_2^2] \cup [v_3^2, v_4^2]$.}\label{dual}
\end{center}
\end{figure}

\section{Gromov hyperbolicity and strong isoperimetric inequalities}\label{SG}
Let $X$ be a geodesic metric space; that is, $X$ is a metric space such
that for every $a, b \in X$ there is a geodesic line
segment $\gamma$ from $a$ to $b$ such that $\mbox{dist}(a,b)= \mbox{length}(\gamma)$.
Then $X$ is called $\delta$-\emph{hyperbolic} if every geodesic triangle $\triangle \subset X$
is $\delta$-\emph{thin}; i.e., any side of $\triangle$ is contained in the $\delta$-neighborhood of
the union of the other two sides. If $X$ is $\delta$-hyperbolic for some $\delta \geq 0$, we just say
that $X$ is hyperbolic in the sense of Gromov, or \emph{Gromov hyperbolic}. For other equivalent
definitions and general theory about Gromov hyperbolic spaces, we refer \cite{Gro, CDP, GH}.
Note that every connected graph can be realized as a geodesic metric space, where the metric is
the simplicial metric such that every edge is of length 1.

Now suppose $\varphi : [\alpha, \beta] \subset \mathbb{R} \to X$ is a path; i.e., a continuous function. We say that
$\varphi$ is $t$-\emph{detour} if there exists a geodesic segment
$\gamma$ from $\varphi(\alpha)$ to $\varphi(\beta)$ and a point $z \in \gamma$ such that
$\mbox{Im} (\varphi) \cap B(z, t) = \emptyset$. Here $B(z,t)$ denotes the open ball with center $z$ and radius $t$.
The detour growth function $g_X : (0, \infty) \to [0, \infty]$ is defined by
\[
 g_X (t) := \inf \{ \mbox{length}(\mbox{Im}\varphi) : \varphi \mbox{ is a } t \mbox{-detour map} \}
\]
with the convention $g_X (t) = \infty$ when there is no rectifiable $t$-detour map. Then it is known \cite{Bonk}
that a geodesic metric space
$X$ is Gromov hyperbolic if and only if $\lim_{t \to \infty} g_X (t) /t = \infty$, which
is what we will use for a proof of Theorem~\ref{TG}.

Suppose $X_1$ and $X_2$ are metric spaces. A function $f: X_1 \to X_2$ is called
a \emph{rough isometry}, or  \emph{quasi-isometry}, if there exist constants $A \geq 1$, $B \geq 0$, and
$C \geq 0$ such that for all $x, y \in X_1$ we have
\[
\frac{1}{A} \, \mbox{dist}(x,y) - B \leq \mbox{dist}\bigl( f(x), f(y) \bigr) \leq A \, \mbox{dist}(x,y) + B,
\]
and for every $w \in X_2$ there exists $x \in X_1$
such that $\mbox{dist}(f(x), w) \leq C$. We say that $X_1$ is roughly isometric to $X_2$ if there exists a rough
isometry from $X_1$ to $X_2$, and it is not
difficult to see that rough isometries define an equivalence relation on the space of metric spaces.
Moreover, it is well known that if $X_1$ is roughly isometric to $X_2$, then $X_1$ is Gromov hyperbolic
if and only if $X_2$ is Gromov hyperbolic (cf. \cite[p.~35]{CDP} or \cite[p.~6]{BSc}),
and if $G_1$ and $G_2$ are roughly isometric graphs of bounded vertex degree, then $\jmath(G_1)>0$
if and only if $\jmath(G_2) >0$  \cite[Theorem~(7.34)]{So2}.
Note that we already used the latter fact in the introduction.

Now we are ready to prove Theorem~\ref{TG}. For (a), the following proof is suggested by Mario Bonk.
Suppose $G$ is a planar graph of bounded face degree
such that $\kappa(G) > 0$, or equivalently, $\jmath(G) > 0$. We assume without loss of generality
that $G$ is a triangulation graph of the plane, since otherwise we can add
bounded number of vertices and edges on every face of $G$ to obtain a triangulation graph $G'$.
This is possible because face degrees of $G$ are bounded.
Then obviously  $G'$ is roughly isometric to $G$,
so we have $G$ is Gromov hyperbolic if and only if $G'$ is Gromov hyperbolic. Moreover, in this case $(G')^*$ is also
roughly isometric to $G^*$, so Theorem~1(b) and Theorem~(7.34) of \cite{So2} imply that $\jmath(G') > 0$
since vertex degrees of $(G')^*$ and $G^*$ are bounded.
(Alternatively, it is not difficult to show that $\kappa(G') > 0$ directly.)

Let $\varphi$ be a $t$-detour map from $a \in G$ to $b \in G$.
By the definition there exist a geodesic  segment $\gamma$ from $a$ to $b$,
and a point $z \in \gamma$ such that $\varphi \cap B(z, t) = \emptyset$,
where we denoted $\mbox{Im}(\varphi)$ by $\varphi$ for simplicity. Furthermore
by shrinking $\gamma$ and $\varphi$ if necessary, we can assume that $\gamma$ meets $\varphi$ only at
$a$ and $b$. Thus we can treat $\gamma \cup \varphi$ as a topological triangle with vertices $a, b$, and $z$.

Let $\Gamma$ be the subgraph of $G$ whose vertices are those on $\gamma$, $\varphi$, and the
bounded component of $G \setminus (\gamma\cup \varphi)$. Then since $G$ is a triangulation graph,
$\Gamma$ becomes a triangulation of the 2-dimensional simplex $\triangle$ with $\partial \triangle = \gamma \cup \varphi$.
Let $\gamma_0 = \gamma \cap B(z, t/4)$,
$\gamma_a$ be the component of $\gamma \setminus \gamma_0$ containing $a$, and
$\gamma_b$ be the component of $\gamma \setminus \gamma_0$ containing $b$. Define
\begin{align*}
A_1 &:= \{ v \in V(\Gamma) : \mbox{dist}(v, \gamma_0) \leq t/4 \mbox{ and } \mbox{dist}(v, \gamma_0)
                              \leq  \mbox{dist}(v, \gamma_a \cup \gamma_b) \},\\
A_2 &:= \{ v \in V(\Gamma) \setminus A_1 : \mbox{dist}(v, \gamma_a) <  \mbox{dist}(v, \gamma_b) \}, \\
A_3 &:= \{ v \in V(\Gamma) \setminus A_1 : \mbox{dist}(v, \gamma_b) \leq  \mbox{dist}(v, \gamma_a) \}.
\end{align*}
It is clear that $\varphi \cap A_1 = \emptyset$. Otherwise there exists $v \in \varphi$ and $w \in \gamma_0$
such that $\mbox{dist}(v, w) \leq t/4$, so we must have
$\mbox{dist}(v, z) \leq \mbox{dist}(v, w) + \mbox{dist}(w, z) \leq t/2$,
contradicting the assumption $\varphi \cap B(z, t) = \emptyset$.
Moreover if $v \in \gamma$ lies between $a$ and $z$, that is, if $v$ lies on the side
opposite the vertex $b$ of $\gamma \cup \varphi$,
then $v$ is definitely in either $A_1$ or $A_2$. Similarly if $v \in \gamma$ lies between $b$ and $z$,
then $v$ belongs to $A_1 \cup A_3$. Thus if we label every vertex $v \in A_i$ by $i \in \{ 1,2,3 \}$,
it becomes so called \emph{Sperner labeling} hence there exists a triangle in $\Gamma$ with vertices
$v_i \in A_i$, $i=1,2,3$, by Sperner's lemma \cite[p.~124]{BSc}.

Let $B:=B(v_1, t/8)$ and  we claim that $B \cap (\varphi \cup \gamma) = \emptyset$ for $t \geq 12$. First,
note that $B \cap \varphi = \emptyset$ since otherwise there exists $v \in \varphi$
such that $\mbox{dist}(v,z) \leq 5t/8$. If $\mbox{dist}(v_1, \gamma) \leq t/8$, then  because
$v_1 \in A_1$ we have $\mbox{dist}(v_1, \gamma_0) \leq t/8$. But because $v_2 \in A_2$ and
$$\mbox{dist}(v_2, \gamma_0) \leq 1 + \mbox{dist}(v_1, \gamma_0) \leq 1 + t/8 \leq t/4$$ for $t \geq 8$,
there exists $x \in \gamma_a$ such that
\[
\mbox{dist}(v_2, x) = \mbox{dist}(v_2, \gamma_a) < \mbox{dist}(v_2, \gamma_0) \leq 1 + t/8.
\]
Similarly there exists $y \in \gamma_b$ such that $\mbox{dist}(v_3, y) \leq 1 + t/8$, so
we have
\begin{equation}\label{vv}
\mbox{dist}(x,y) \leq\mbox{dist}(x,v_2) + \mbox{dist}(v_2,v_3) + \mbox{dist}(v_3,y) < 3 + t/4.
 \end{equation}
On the other hand, $\gamma$ is a geodesic segment and $B(z, t/4)$
separates $\gamma_a$ and $\gamma_b$. Then
because $x \in \gamma_a$ and $y \in \gamma_b$,
we must have $\mbox{dist}(x,y) \geq t/2$.
This contradicts \eqref{vv} for $t \geq 12$, so the claim follows.

Note that
\[
|\partial_e \Gamma| = \mbox{length}(\varphi) + \mbox{length}(\gamma) \leq 2 \cdot \mbox{length}(\varphi).
\]
Moreover because $v_1 \in V(\Gamma)$ and $B \cap (\varphi \cup \gamma) = \emptyset$ for $t \geq 12$,
we must have $B \subset \Gamma$ in this case. Thus if $t \geq 12$ and $8n \leq t < 8(n+1)$ for some $n \in \mathbb{N}$, then
\[
|V(\Gamma)| \geq |V(B)| \geq \bigl( 1+\jmath(G) \bigr)^n \geq \bigl( 1+\jmath(G) \bigr)^{t/8-1}.
\]
Now because $|\partial_e \Gamma| \geq |\partial_v \Gamma|$, we have
\[
\mbox{length}(\varphi) \geq \frac{1}{2} |\partial_e \Gamma| \geq \frac{1}{2} |\partial_v \Gamma|
\geq \frac{1}{2}\jmath(G) |V(\Gamma)| \geq \frac{1}{2}\jmath(G)\bigl( 1+\jmath(G) \bigr)^{t/8-1}.
\]
This proves that $g_G (t) / t \to \infty$ as $t \to \infty$, where $g_G$ is the detour growth function
introduced at the beginning of this section. We conclude that $G$ is Gromov hyperbolic by \cite{Bonk}.

We remark here that the above argument can be considered an alternative proof of (a part of) Theorem~2.1 in
\cite[Chap.~6]{CDP}, where it is proved that every reasonable complete simply connected Riemannian manifold
satisfying a linear isoperimetric inequality must be Gromov hyperbolic. In fact, one only needs to
modify some terminologies in the above proof so that they are adequate to the continuous case, and use the continuous
version of Sperner's lemma \cite[p.~378]{AH} instead of the combinatorial version. Since it is irrelevant to
our subject, we omit the details here.

To prove Theorem~\ref{TG}(b), let us construct a normal planar graph $G_1$ which is Gromov hyperbolic
but  $\jmath(G_1)=\kappa(G_1)=0$ . Let $\Gamma_1$ be a triangulation of the plane such that $\deg v \geq 7$ for all
$v \in V(\Gamma_1)$. Then obviously $\Gamma_1$ is Gromov hyperbolic.
Choose a sequence of edges $e_n =[a_n, b_n] \in E(\Gamma_1)$ which are far away from each other,
for instance $\mbox{dist}(a_n, a_m) \geq 10$ for $n \ne m$,
and let $f_n$ and $g_n$ be the triangular faces of $\Gamma_1$ sharing  $e_n$. Also let
$c_n$ and $d_n$ be the vertices of $f_n$ and $g_n$, respectively, which are not lying  on $e_n$.

We obtain a new graph from $\Gamma_1$ by replacing $e_n$ by $n$-multiple edges
and drawing a line from $c_n$ to $d_n$ (Figure~\ref{gromov2}). We do this operation for all $n \in \mathbb{N}$, and let
$G_1$ denote the resulting graph. $G_1$ is obviously normal. Moreover,  $G_1$ is Gromov hyperbolic
because it is roughly isometric to $\Gamma_1$. It is also easy to see that $\jmath(G_1)=\kappa(G_1)=0$, since if we let
$S_n$ be the subgraph of $G_1$ such that $V(S_n)$ consists of $a_n, b_n, c_n, d_n$, and all the added vertices
in $f_n \cup g_n$, then we have $|F(S_n)| = 2n+2$ and $|\partial_e S_n| =4$. This completes the proof
of Theorem~\ref{TG}(b).
\begin{figure}[hbt]
\begin{center}
 \input{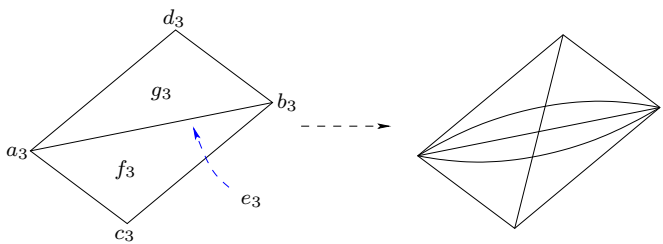}
 \caption{replacing $e_3$ with triple edges and drawing a diagonal line}\label{gromov2}
\end{center}
\end{figure}

The graph $G$ constructed in Section~\ref{Scur} also satisfies the properties in
Theorem~\ref{TG}(b) except being normal. To explain this, first note that we could construct $G$ so that every face is of degree
at most 4. Now if we let $S$ be the \emph{infinite} subgraph of $G$ such that
\[
V(S) = V(G) \setminus \{ v_k^n : n \in \mathbb{Z} \mbox{ and } 1 \leq k \leq 2|n| \},
\]
then one can check that $S$ is a planar graph of bounded face degree such that every vertex has degree at least 7. Thus
$S$ is Gromov hyperbolic by Corollary~\ref{Cor} and Theorem~\ref{TG}, so $G$ is also
Gromov hyperbolic since it is roughly isometric to $S$.

For the last, we prove Theorem~\ref{TG}(c); i.e., construct a normal planar graph $G_2$ of bounded
face degree such that $\imath(G_2)>0$ but not Gromov hyperbolic. The main idea here is to construct a graph
with a structure far from being hyperbolic, but the simple random walk on it is transient.

We start with the square lattice graph $\Gamma_2$; i.e., $V(\Gamma_2) = \mathbb{Z} \times \mathbb{Z}$,
and $\Gamma_2$ has an edge between $(n_1, m_1)$ and $(n_2, m_2)$ if and only if $|n_1 - n_2| + |m_1 - m_2| =1$.
Let $O=(0,0)$ be the origin, and for each $n \in \mathbb{N}\cup\{0\}$ define $V_n$
as the set of vertices of $\Gamma_2$ whose combinatorial distance from $O$ is equal to $n$. Also for each $n \in \mathbb{N}$, let
 $E_n$ be the set of edges of $\Gamma_2$ connecting a vertex
in $V_{n-1}$ to another one in $V_n$.
Then as in Theorem~\ref{TG}(b) we replace each edge in $E_n$ by $\ell_n$-multiple edges, where
the sequence $\{ \ell_n \}$ increases to infinity very fast but will be determined later.
Finally we draw the lines $x=m + 1/2$ and $y = m+1/2$
for all $m \in \mathbb{Z}$. We can do this operation so that each vertical line $x=m+1/2$ meets
no vertical edges, and meets every horizontal edge at most once. Of course
 we can draw the horizontal lines $y= m+1/2$ in a similar way, and let $G_2$ denote the obtained graph (Figure~\ref{gromov3}).

\begin{figure}[t!]
\begin{center}
 \input{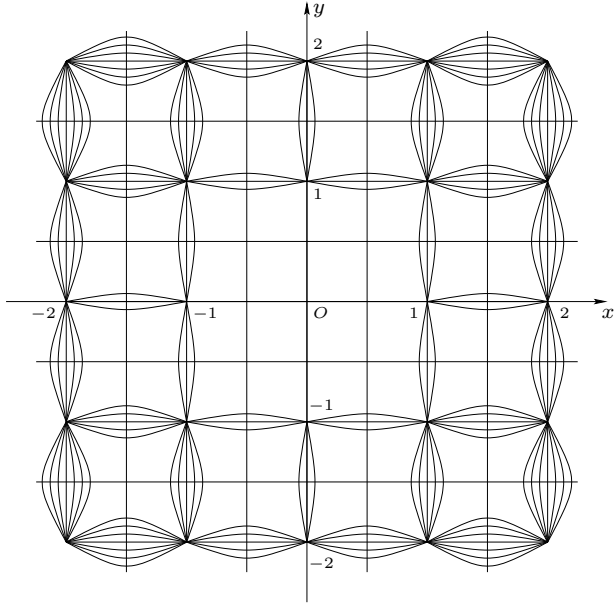}
 \caption{the graph $G_2$; it is drawn with $\ell_n =2n-1$ for aesthetic reasons, but $\ell_n$ should
 increase faster}\label{gromov3}
\end{center}
\end{figure}

It is clear that $G_2$ is not Gromov hyperbolic, since it is roughly isometric to $\Gamma_2$.
Also one can immediately see that $G_2$ is normal, face degrees of $G_2$ are bounded above by 4,
and $\jmath(G_2) = \kappa(G_2) =0$.
So we only need to show that $\imath(G_2)>0$.
Define $h:V(\Gamma_2) \to V(G_2)$ as the natural injection which maps each integer lattice
point to itself, and suppose that $S$ is a finite subgraph of $G_2$. Let $N$ be the largest natural number
such that $h(V_N) \cap V(S) \ne \emptyset$,
and $W_1$ be the set of vertices of $S$ which lies in the region
$\{(x,y): |x| + |y|<N+1/10\}$. If no such $N$ exists, we set $W_1 = \emptyset$.
We also define $W_2$ as the set of vertices in $V(S) \setminus W_1$ lying
on the intersection of two lines $x= k_1+1/2$ and $y=k_2+1/2$ for some $k_1, k_2 \in \mathbb{Z}$, and let
$W_3 = V(S) \setminus (W_1 \cup W_2)$.

The vertices in $W_3$ must be on the `middle' of some multiple edges, so each of them has a neighbor in $V_{l}$
for some $l \geq N+1$.
Consequently $W_3 \subset \partial_v S$ and $|W_3| \leq |\partial_v S| \leq |\partial S|$. For $W_2$,
the only neighbors of a vertex $v \in W_2$ are those at the `middle' of some multiple edges.
Then since $(k_1 + 1/2)+(k_2+1/2) \geq N+1/10$ implies $k_1 + k_2 + 1/2 \geq N+1/10$ for $k_1, k_2 \in \mathbb{N}$,
every $v \in W_2$ is either in $\partial_v S$ or has a neighbor in $W_3$. Thus
$|W_2| \leq |W_3| + |\partial_v S| \leq 2 |\partial S|$.

Suppose $W_1 \ne \emptyset$, and let $v \in h(V_N) \cap V(S)$.
Also let us assume that there are $k$ edges in $\partial S$ with one end
at $v$. Then by our construction
$v$ must have at least  $\ell_{N+1} - k$ neighbors in $W_3$, all of which are in $\partial_v S$.
So we must have $|\partial S| \geq \ell_{N+1}$. On the other hand, if we denote by
$C_n$ the number of vertices in $G_2$ lying in the region $\{(x,y): |x| + |y|<n+1/10\}$,
then definitely $C_n$ depends only on $\ell_1, \ldots, \ell_n$ and $n$.
Thus we can choose the sequence $\{ \ell_n \}$ so that $\ell_{n+1} \geq C_n$ for all $n \in \mathbb{N}$.
This in particular implies that $|W_1| \leq C_N \leq \ell_{N+1} \leq |\partial S|$, and note that
the inequality $|W_1| \leq |\partial S|$ is obviously true when $W_1 = \emptyset$.

So far we have shown that $|V(S)|=|W_1|+|W_2|+|W_3| \leq 4 |\partial S|$, but this is enough to conclude
$\imath(G_2) >0$ by considering the duality property of Lemma~\ref{main lemma}. This completes the proof
of Theorem~\ref{TG}.

\section{Further remarks}
One of the main assumptions of Theorem~\ref{T}(b) is that $G$ is embedded into the plane locally finitely,
so it must be a planar graph with only \emph{one end}. Recently we have extended this result
to the case when $G$ has finitely many ends \cite{OS}. Furthermore, if $G$ is normal,
Theorem~\ref{T}(b) remains valid even when $G$ has infinitely many ends.


\begin{thebibliography}{99}
\bibitem{AH}
P. Alexandroff and H. Hopf, \emph{Topologie I}, Berichtigter Reprint,
Die Grundlehren der mathematischen Wissenschaften, Band 45, Springer-Verlag, Berlin-New York, 1974.

\bibitem{BP1}
O. Baues and N. Peyerimhoff, \emph{Curvature and geometry of tessellating plane graphs,}
Discrete Comput. Geom. 25 (2001), no. 1, 141--159.

\bibitem{BP2}
O. Baues and N. Peyerimhoff, \emph{Geodesics in non-positively curved plane tessellations,}
Adv. Geom. 6 (2006), no. 2, 243--263.

\bibitem{BMeS}
I. Benjamini, S. Merenkov, and O. Schramm,
\emph{A negative answer to Nevanlinna's type question and a parabolic surface with a lot of negative curvature,}
Proc. Amer. Math. Soc. 132 (2004), no. 3, 641--647.

\bibitem{BS}
I. Benjamini and O. Schramm,
\emph{Every graph with a positive Cheeger constant contains a tree with a positive Cheeger constant},
Geom. Funct. Anal. 7 (1997), no. 3, 403--419.

\bibitem{BMS}
N. Biggs, B. Mohar, and J. Shawe-Taylor, \emph{The spectral radius of infinite graphs,}
Bull. London Math. Soc. 20 (1988), no. 2, 116--120.

\bibitem{Bonk}
M. Bonk, \emph{Quasi-geodesic segments and Gromov hyperbolic spaces,}
Geom. Dedicata 62 (1996), no. 3, 281--298.

\bibitem{BSc}
S. Buyalo and V. Schroeder,
\emph{Elements of asymptotic geometry},
EMS Monographs in Mathematics. European Mathematical Society (EMS), Zurich, 2007.

\bibitem{Che}
J. Cheeger, \emph{A lower bound for the smallest eigenvalue of the Laplacian. Problems in analysis,} pp. 195--199, Princeton Univ. Press, Princeton, N. J., 1970.

\bibitem{CDP}
M.~Coornaert, T.~Delzant and A.~ Papadopoulos, \emph{G\'eom\'etrie et
th\'eorie des groupes}, LNM, Vol.~1441, Springer, Berlin, 1990.

\bibitem{Cor}
J. Corson, \emph{Conformally nonspherical 2-complexes,}
Math. Z. 214 (1993), no. 3, 511--519.

\bibitem{DM}
M. DeVos and B. Mohar, \emph{An analogue of the Descartes-Euler formula for infinite graphs and Higuchi's conjecture,} Trans. Amer. Math. Soc. 359 (2007), no. 7, 3287--3300.

\bibitem{Dod}
J. Dodziuk, \emph{Difference equations, isoperimetric inequalities and transience of certain random walks,}
Trans. Amer. Math. Soc. 284 (1984), no. 2, 787--794.

\bibitem{DK}
J. Dodziuk and W. Kendall, \emph{Combinatorial Laplacians and isoperimetric inequality,} From local times to global geometry, control and physics (Coventry, 1984/85), 68--74,
Pitman Res. Notes Math. Ser., 150, Longman Sci. Tech., Harlow, 1986.

\bibitem{Fu}
K. Fujiwara, \emph{The Laplacian on rapidly branching trees,}
Duke Math. J. 83 (1996), no. 1, 191--202.

\bibitem{Gerl}
P. Gerl, \emph{Random walks on graphs with a strong isoperimetric property,}
J. Theoret. Probab. 1 (1988), no. 2, 171--187.

\bibitem{GH}
 E.~Ghys and P.~de~la~Harpe (eds.), \emph{Sur les Groupes Hyperbolique d'apr\`es
 Mikhael Gromov}, Birkh\"auser, Boston, MA, 1990.

\bibitem{Gro}
 M.~Gromov, \emph{Hyperbolic Groups}, In: Essays in Group theory (S.~Gersten eds.),
 MSRI Publication 8, Springer, 1987, 75--263.

\bibitem{GS}
B. Gr\"{u}nbaum and G. Shephard,
\emph{Tilings and patterns}, W. H. Freeman and Company, New York, 1987.

\bibitem{Hig}
Y. Higuchi, \emph{Combinatorial curvature for planar graphs,}
J. Graph Theory 38 (2001), no. 4, 220--229.

\bibitem{HS}
Y. Higuchi and T. Shirai, \emph{Isoperimetric constants of $(d,f)$-regular planar graphs,}
Interdiscip. Inform. Sci. 9 (2003), no. 2, 221--228.

\bibitem{Kel1}
M. Keller, \emph{The essential spectrum of the Laplacian on rapidly branching tessellations,}
Math. Ann. 346 (2010), no. 1, 51--66.

\bibitem{Kel2}
M. Keller, \emph{Curvature, geometry and spectral properties of planar graphs,}
 Discrete Comput. Geom. 46 (2011), no. 3, 500--525.

\bibitem{KP}
M. Keller and N. Peyerimhoff, \emph{Cheeger constants, growth and spectrum of locally tessellating planar graphs,} Math. Z. 268 (2011), no. 3-4, 871--886.

\bibitem{LPZ}
S. Lawrencenko, M. Plummer, and X. Zha, {Isoperimetric constants of infinite plane graphs},
Discrete Comput. Geom. 28 (2002), no. 3, 313--330.

\bibitem{Moh}
B. Mohar, \emph{Isoperimetric numbers and spectral radius of some infinite planar graphs,}
Math. Slovaca 42 (1992), 411--425.

\bibitem{Nev}
R. Nevanlinna, \emph{Eindeutige analytische Funktionen,}
Springer-Verlag, 1936 (and also 1974). Translated as \emph{Analytic Functions}, Die Grundlehren der mathematischen Wissenschaften, Band 162, Springer-Verlag, 1970.

\bibitem{Oh}
B. Oh,  \emph{Aleksandrov surfaces and hyperbolicity},
Trans. Amer. Math. Soc. \textbf{357} (2005), no. 11, 4555--4577.

\bibitem{OS}
B. Oh and J. Seo, \emph{Duality properties of strong isoperimetric inequalities on a planar graph with
more than one end}, in preparation.

\bibitem{So1}
P. Soardi, \emph{Recurrence and transience of the edge graph of a tiling of the Euclidean plane,}
Math. Ann. 287 (1990), no. 4, 613--626.

\bibitem{So2}
 P.~Soardi, \emph{Potential theory on infinite networks,} LNM 1590,
 Springer-Verlag, Berlin, 1994.

\bibitem{Sto}
D. Stone, \emph{A combinatorial analogue of a theorem of Myers,}
Illinois J. Math. 20 (1976), no. 1, 12--21.

\bibitem{Tei}
O. Teichm\"{u}ller, \emph{Untersuchungen \"{u}ber konforme und quasikonforme Abbildung,}
Deutsch. Math. 3 (1938), 621--678.

\bibitem{Woe}
 W. Woess, \emph{A note on tilings and strong isoperimetric inequality},
 Math. Proc. Cambridge Philos. Soc. 124 (1998), no. 3, 385--393.

\bibitem{Woe2}
W. Woess,  \emph{Random walks on infinite graphs and groups,}
Cambridge Tracts in Mathematics, 138. Cambridge University Press, Cambridge, 2000.

\bibitem{Zuk}
A. \.{Z}uk, \emph{On the norms of the random walks on planar graphs,}
 Ann. Inst. Fourier (Grenoble) 47 (1997), no. 5, 1463--1490.
\end{thebibliography}
\end{document}